\documentclass[12pt]{amsart}
\usepackage{amssymb}
\usepackage{color}
\newtheorem{thm}{Theorem}

\newtheorem{lem}[thm]{Lemma}
\newtheorem{prop}[thm]{Proposition}
 \newtheorem{conj}[thm]{Conjecture}

 \theoremstyle{definition}

 \theoremstyle{remark}

\newtheorem{example}[thm]{Example}

\usepackage{caption}
\usepackage{subcaption}
\usepackage{verbatim}
\newcommand{\TryPackage}[3]{\IfFileExists{#1.sty}{\usepackage{#1}#2}{#3}}
\TryPackage{mathrsfs}{\renewcommand{\mathcal}{\mathscr}}{%
     \TryPackage{eucal}{}{}}

\newcommand*\wbar[1]{
  \hbox{ \kern0.01em%
    \vbox{%
      \hrule height 0.5pt  
      \kern0.25ex
      \hbox{%
        \kern-0.2em
        \ensuremath{#1}%
        \kern-0.05em
      }%
    }%
  \kern-0.08em}%
}

\newcommand{\wt}{\widetilde}

\newcommand{\ep}{\varepsilon}

\renewcommand{\phi}{\varphi}

\newcommand{\Om}{\Omega}

\newcommand{\RR}{\mathbb{R}}

\newcommand{\ZZ}{\mathbb{Z}}
\newcommand{\cE}{\mathcal{E}}

\newcommand{\vb}{\operatorname{vb}}
\newcommand{\vbr}{\operatorname{vbr}}
\newcommand{\wb}{\operatorname{wb}}
\newcommand{\br}{\operatorname{br}}
\newcommand{\mr}{\operatorname{mr}}

\newcommand{\RG}{\operatorname{\it \wbar G}}
\newcommand{\QG}{\operatorname{\it QG}}

\newcommand{\sm}{\smallsetminus}

\usepackage{bbm}

\usepackage{enumerate}

\usepackage{rotating}

\usepackage{graphicx}
\usepackage{amssymb}
\usepackage{epstopdf}
\usepackage{amsmath}

\begin{document}

\title[Bridge numbers for virtual and welded knots]
{Bridge numbers  for virtual and welded knots}
\author[Boden]{Hans U. Boden}
\address{Mathematics \& Statistics, McMaster University, Hamilton, Ontario}
\email{boden@mcmaster.ca}

\author[Gaudreau]{Anne Isabel Gaudreau}
\address{Mathematics \& Statistics, McMaster University, Hamilton, Ontario}
\email{gaudreai@mcmaster.ca}


\begin{abstract} Using Gauss diagrams, one can define the virtual bridge number $\vb(K)$ and the welded bridge number $\wb(K)$, invariants of virtual and welded knots satisfying $\wb(K) \leq \vb(K)$. If $K$ is a classical knot, Chernov and Manturov showed that $\vb(K)=\br(K)$, the bridge number as a classical knot, and we ask whether the same thing is true for welded knots. The welded bridge number is bounded below by the meridional rank of the knot group $G_K$, and we use this to relate this question to a conjecture of Cappell and Shaneson. We show how to use other virtual and welded invariants to further investigate bridge numbers. Among them are Manturov's parity and the reduced virtual knot group $\RG_K$, and we apply these methods to address Questions 6.1, 6.2, 6.3 \& 6.5 raised by Hirasawa, Kamada and Kamada in their paper \emph{Bridge presentation of virtual knots}, J. Knot Theory Ramifications {\bf 20} (2011), no. 6, 881Ð-893.
 \end{abstract}

\maketitle
\section{Virtual and welded knots}

The classical theory of knots involves studying embeddings of $S^1\hookrightarrow S^3$ up to ambient isotopy. Under projection to a generic two-plane, any given knot is mapped to an immersion $S^1 \looparrowright \RR^2$ with finitely many immersion points, each of which is a transverse double point. The 3-dimensional knot can be reconstructed from its projection by keeping track of crossing information at each double point. The knot projection, together with the crossing information, is called the knot diagram. Two knot diagrams correspond to isotopic knots if and only if they are related by a sequence of Reidemeister moves. 

To every oriented knot diagram corresponds a Gauss word, which records the order in which the crossings are encountered starting from a fixed basepoint on the knot, along with the sign of each crossing. The Gauss word is conveniently represented by a decorated trivalent graph called the Gauss diagram, consisting of a circle oriented counterclockwise together with $n$ signed, directed arcs connecting $n$ distinct pairs of points on the circle. The directed arcs, or chords, point from an overcrossing (arrowtail) to an undercrossing (arrowhead). The sign of the chord indicates whether the crossing is right-handed (positive) or left-handed (negative). The Reidemeister moves on knot diagrams can be translated into corresponding moves on the Gauss diagrams. In \cite{Polyak}, Polyak shows that the complete set of Reidemeister moves on oriented diagrams can be generated from the four moves $\Om 1a, \Om 1b, \Om 2a,$ and $\Om 3a$ illustrated in Figure \ref{RM}. 

Although every classical knot diagram is determined by its Gauss diagram, not all Gauss diagrams come from classical knots. We call a Gauss diagram \emph{planar} if it represents a classical knot diagram. By inserting virtual crossings, which are displayed as double points with circles around them, one can draw a virtual knot diagram corresponding to any given Gauss diagram, whether or not it is planar.

Virtual knot theory involves studying virtual knot diagrams up to virtual isotopy. Note that Reidemeister moves of Figure \ref{RM} �do not assume the underlying diagrams are planar, and we say that two Gauss diagrams $D_1$ and $D_2$ are \emph{virtually isotopic} if there is a sequence of  Reidemeister moves that transform $D_1$ into $D_2$. Virtually isotopy defines an equivalence relation on Gauss diagrams, and virtual knots are defined to be the equivalence classes of Gauss diagrams under virtual isotopy. Therefore, if two planar diagrams are classically isotopic, then they must be virtually isotopic. The converse to this statement, while true, is not at all obvious. The easiest way to explain that is to introduce the knot group $G_K$ associated to a virtual knot $K$.

\begin{figure}[ht]
\centering
\includegraphics[scale=0.70]{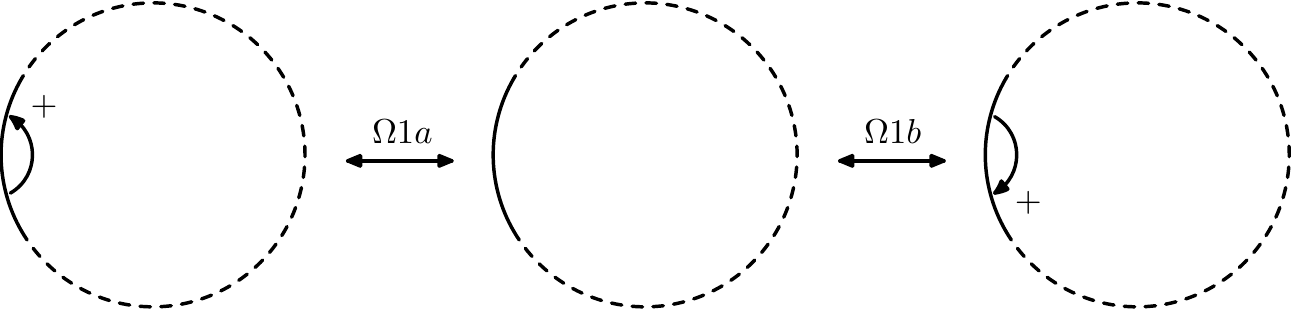} \smallskip \\
\includegraphics[scale=0.70]{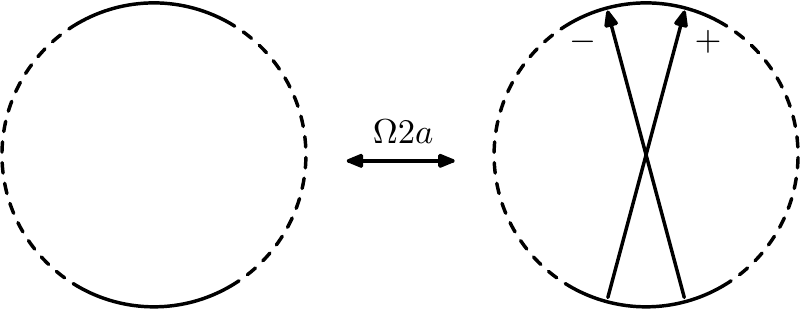}  \bigskip \\
 \includegraphics[scale=0.70]{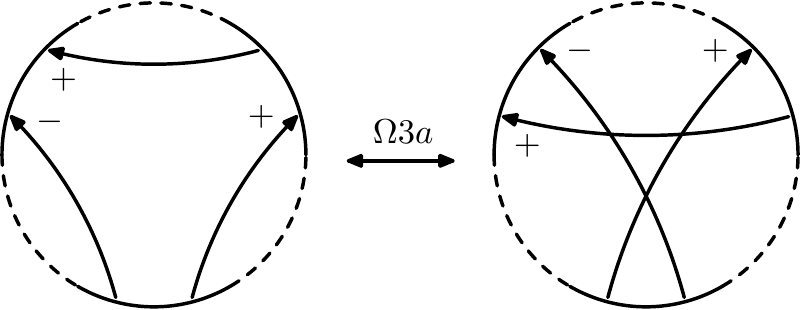}  \qquad  
  \includegraphics[scale=0.70]{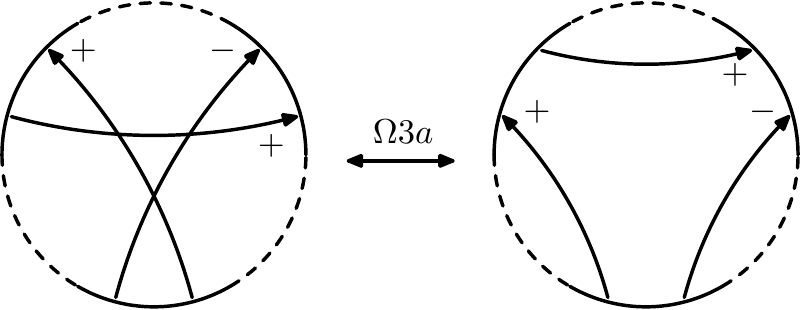}
\caption{Polyak's minimal generating set of Reidemeister moves.}
\label{RM}
\end{figure}

Given a Gauss diagram $D$, the knot group $G_D$ is the finitely presented group given by the following construction. Suppose $D$ has $n$ chords, and consider the generators $a_1, \ldots, a_n$ given by the $n$ arcs of $D$ from one arrowhead to the next. We order the chords $c_i$ and arcs $a_i$ consistently so that the arrowhead of $c_i$ separates $a_i$ from $a_{i+1}$, modulo $n$. Suppose an arrowtail of $c_i$ lies on the arc labeled by $a_j$ as in Figure \ref{GD-crossing}. In this case, we impose the relation $a_{i+1} = a_j^{\ep_i} a_i a_{j}^{-\ep_i}$, where $\ep_i$ is the sign of $c_i$.

\begin{figure}[ht]
\centering
\includegraphics[scale=0.90]{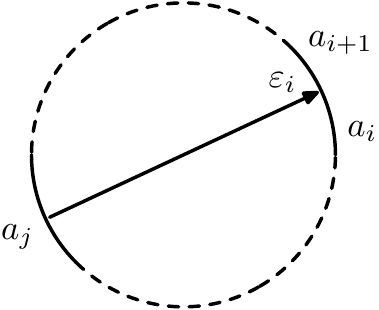}  
\caption{The chord $c_i$ gives the Wirtinger relation $a_{i+1} = a_j^{\ep_i} a_i a_{j}^{-\ep_i}$.}
\label{GD-crossing}
\end{figure}

The group $G_D$ is invariant under virtual isotopy, and the resulting invariant is called the knot group of $K$ and is denoted $G_K$. Notice that if $K$ is classical, then $G_K$ is just the fundamental group of the complement $S^3 \sm \tau K$ of a tubular neighbourhood of $K$. For any virtual knot $K$, one can define commuting elements $\mu, \lambda \in G_K$ called the meridian and longitude, and the peripheral structure determined by the pair $\mu, \lambda$ is again invariant under virtual isotopy. Now for classical knots, Waldhausen showed that the knot group $G_K$ and its peripheral structure form a complete invariant of $K$, and since both are invariant under virtual isotopy, Waldhausen's theorem implies that two classical knot diagrams are virtually isotopic if and only if they are classically isotopic.  This argument shows that classical knot theory embeds into virtual knot theory. 

In welded knot theory, we say two Gauss diagrams are \emph{welded equivalent} if they are related by a sequence of  Reidemeister moves and the forbidden overpass, which is a move that allows one to switch adjacent arrowtails, see Figure \ref{GD-forbidden}. A welded equivalence class of Gauss diagrams is called a welded knot, and it can also be regarded as an equivalence class of virtual knots. The knot group $G_K$ and peripheral structure are invariant under welded equivalence, and thus Waldhausen's theorem shows that classical knots embed into the coarser theory of welded knots. 

\begin{figure}[ht]
\centering
\includegraphics[scale=0.90]{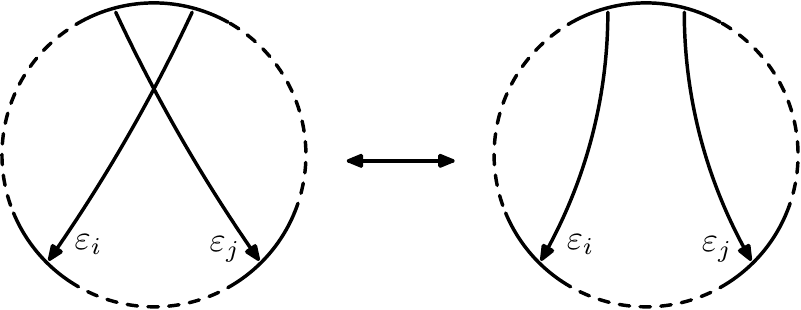}  
\caption{The forbidden overpass.}
\label{GD-forbidden}
\end{figure}

\section{Virtual and welded bridge numbers}

In this section, we introduce invariants of virtual and welded knots based on the bridge number of the underlying Gauss diagrams.
We begin with the definition of the virtual bridge number $\vb(K)$.
Given a Gauss diagram $D$, an overbridge of $D$ is defined to be a maximal arc of the circle containing only arrowtails (overcrossings) on it,  
and the bridge number of $D$ is the number of overbridges of $D$, denoted $\vbr(D)$. Given a virtual knot $K$, we define its virtual bridge number $\vb(K)$ to be the minimum bridge number taken over all Gauss diagrams $D$ virtually isotopic to $K$. By convention, the unknot is taken to have 
bridge number $1$, cf. \cite{HKK}.

For a classical knot $K$, its classical bridge number $\br(K)$ is defined similarly as the minimum bridge number taken over all planar Gauss diagrams representing $K$. Since $\br(K)$ is a minimum over a smaller set of Gauss diagrams, we see that $\vb(K) \leq \br(K)$. It is natural to ask (\cite[Question 6.4]{HKK}, \cite{BC}) if there are examples where this inequality can be strict. In \cite[Theorem 3.2]{Chernov} and \cite[Corollary 2]{Manturov}, Chernov and Manturov answer this question using a projection from virtual knots to classical knots defined using parity. This is summarized in the following theorem,
and we will discuss parity and its application to virtual bridge numbers in Section \ref{PandP}.  

\begin{thm}[Chernov, Manturov] \label{manturov}
If $K$ is a classical knot, then $\vb(K) = \br(K).$
\end{thm}
 
In an analogous way, one can define the welded bridge number $\wb(K)$ as the minimum bridge number over all Gauss diagrams welded equivalent to $K$. Since a welded knot can be viewed as an equivalence class of virtual knots, for any virtual knot $K$, we have
$$\wb(K) ~\leq ~\vb(K).$$
The above inequality can be strict, indeed in Section \ref{section-3} we will see an example of a virtual knot $K$ with $\vb(K)>1$ which is welded equivalent to the unknot. 

Notice that if $K$ is an $m$-bridge knot (virtual or classical), then the knot group $G_K$ can be presented with $m$ meridional generators. To see this, label the overbridges $a_1, \ldots, a_m$ and travel along the knot to the first undercrossing. The over crossing arc is part of a bridge and so is labelled by $a_j$ for some $j$, and the incoming arc is labelled by $a_i$ for some $i$ as well. Thus, the Wirtinger relation implies the outgoing arc is given by $a_j^\ep a_i a_j^{-\ep}$. Continuing in this manner, one can label all the remaining arcs of the diagram of $K$ as words in $a_1, \ldots, a_m$, and it follows
that $G_K$ admits a presentation with $m$ generators corresponding to the  meridians $a_1, \ldots, a_m.$

We define the meridional rank $\mr (G_K)$ of the knot group to be the minimum number of meridional generators taken over all presentations of $G_K$. Since $G_K$ is invariant under welded equivalence, the above construction shows that $\mr (G_K) \leq \wb(K)$ in general. The following conjecture, which is due Cappell and Shaneson (cf. Problem 1.11 of \cite{Kirby}), asks whether every classical knot with meridional rank $m$ is an $m$-bridge knot. 

\begin{conj}[Cappell-Shaneson] \label{CSConj} 
If $K$ is classical, then $\br(K)=\mr(G_K)$. 
\end{conj}

Assuming this conjecture, we obtain the following generalization of Theorem \ref{manturov}.

\begin{thm} \label{thm1}
If Conjecture \ref{CSConj} is true, then $\wb(K) = \br(K)$ for all classical knots. 
\end{thm}

\begin{proof}
For any virtual knot $K$, we know that $G_K$ is invariant under welded equivalence and that $\mr(G_K) \leq  \wb(K)$. Now suppose $K$ is classical. Then Conjecture \ref{CSConj} implies that
$\br(K)=\mr(G_K)$. The conclusion now follows since $\wb(K) \leq \br(K)$ by definition.
\end{proof}
 
Notice that Theorem \ref{thm1} implies Theorem \ref{manturov}, since for $K$ classical one always has $$\wb(K) ~\leq ~\vb(K)~ \leq~ \br(K).$$

We will use the elementary ideal theory of the knot group $G_K$ to deduce information about its meridional rank $\mr(G_K).$ 
 
Consider the knot group $G_K$ of a virtual knot $K$. Let $G_K'$ and $G_K''$ be its first and second commutator subgroups. The abelianization $G_K / G_K'$ is isomorphic to the infinite cyclic group $\ZZ$, and the quotient $G_K'/ G_K''$ is a module over the ring $\ZZ[t^{\pm 1}]$
of  Laurent polynomials called the Alexander module. A presentation matrix for the Alexander module is obtain by Fox differentiation as follows.
Given the group presentation
\begin{equation}\label{Gppres}
G_K = \langle a_1, \ldots , a_n \mid  r_1, \ldots, r_n  \rangle,
\end{equation} 
the associated Alexander matrix  is the $n \times n$ matrix 
\begin{equation}\label{Alex}
A=\left(\left.\frac{\partial r_i}{\partial a_j}\right|_{a_1, \ldots, a_n=t} \right)
\end{equation} 
whose $(i,j)$ entry is the Laurent polynomial
obtained by taking the Fox derivative  $\frac{\partial r_i}{\partial a_j}$ and substituting $t$ for each $a_i$. (Here, we are assuming that each $a_i$ is a meridional generator for $G_K$.)   The matrix $A$ is not uniquely determined by the module, but the associated sequence 
\begin{equation}\label{chain}
 (0)= \cE_0 \subset \cE_1 \subset \cdots \subset \cE_n  = \ZZ[t^{\pm 1}]
 \end{equation}
of elementary ideals  
depends only on the underlying Alexander module. Here, $\cE_k$ is the ideal generated by all $(n-k) \times (n-k)$ minors of $A$ and is called $k$-th elementary ideal of $G_K$. 
(That $\cE_0 =(0)$ follows from the fundamental identity and the fact that the presentation matrix $A$ is square, see \cite[(2.3)]{Fox1953}.) Notice that the length of the chain \eqref{chain} of elementary ideals is bounded by the size of the Alexander matrix \eqref{Alex}. For instance, if $G_K$ has meridional rank $k$, then using a presentation \eqref{Gppres} of $G_K$ with $k$ generators, then
the associated Alexander matrix $A$ is a $k \times k$ matrix, and it follows that $\cE_k = (1)$. The following proposition summarizes the discussion.
\begin{prop}\label{group-mr}
 If $K$ is a virtual knot whose $k$-th elementary ideal $\cE_k \neq (1)$ is proper, then the knot group $G_K$ has meridional rank $\mr(G_K) \geq k+1$ and $K$ has welded bridge number $\wb(K) \geq k+1$.
\end{prop}

\begin{example}
The virtual knot  $K = 6.87262$ has knot group
$$G_K =  \langle a, b, c \mid a^{-1}c^{-1}a^{-1}cac, \, b^{-1}a^{-1}b^{-1}aba, \, c^{-1}b^{-1}c^{-1}bcb \rangle.$$
This knot is almost classical, meaning that it admits an Alexander numbering, and it follows from \cite{NNST} that the first elementary ideal $\cE_1$ is principal. Using the presentation above, one can show directly that $\cE_1$ is generated by $(t^2-t+1)^2$ and further that $\cE_2$ is also principal and generated by $t^2-t+1.$ Since $G_K$ admits a presentation with 3 meridional generators and $\cE_2 \neq (1)$, Proposition \ref{group-mr} implies that $
\mr(G_K) =3.$ Since $K$ can be represented by a diagram with three bridges, it follows that $\vb(K)=\wb(K) = 3.$
\end{example}

\begin{figure}[ht]
\centering
\includegraphics[scale=0.90]{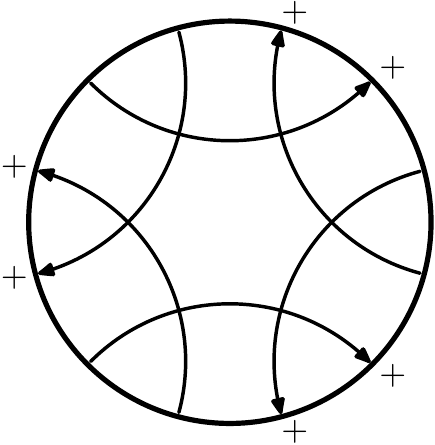}
\caption{The Gauss diagram for the virtual knot $K=6.87262$ with $\vb(K)=\wb(K)=3$.}
\label{6-87262}
\end{figure}

\section{Bridge numbers for connected sums and mirror images} \label{section-3}

If $K_1$ and $K_2$ are classical knots, then under connected sum the bridge number satisfies 
\begin{equation} \label{connsum}
\br(K_1 \# K_2) = \br(K_1)+\br(K_2)-1.
\end{equation}
 
Question 6.5 of \cite{HKK}  asks whether there is a similar formula for virtual knots, i.e. given virtual knots $K_1$ and $K_2$, is it true  that $\vb(K_1\# K_2)=\vb(K_1)+\vb(K_2)-1$? 
We will show by example that this equation does not always hold. In general, given two Gauss diagrams $D_1$ and $D_2$ one can deduce that   
\begin{equation}\label{connect}
\vbr(D_1\# D_2) ~\leq~\vbr(D_1)+\vbr(D_2),
\end{equation}
using the bridge reduction lemma of Hirasawa et.~al., see \cite[Lemma 4.1]{HKK}. Recall that the connected sum in virtual knot theory depends on where the diagrams are connected, and there are numerous different cases to consider. Connecting between undercrossings on $D_1$ and $D_2$, one can see that $\vbr(D_1\# D_2) ~\leq~\vbr(D_1)+\vbr(D_2)$, and similarly for overcrossings. Connecting between overcrossings on one diagram and between undercrossings on the other formally adds a bridge, but since that overbridge does not interact with the underbridge from the other diagram, one can reduce the bridge number using the bridge reduction lemma. The remaining possibilities are if the points of connection are situated between an undercrossing and an overcrossing. If similar bridges are connected to each other, $\vbr(D_1\# D_2) ~\leq~\vbr(D_1)+\vbr(D_2)-1$, otherwise, the bridge reduction lemma can be applied twice to get the same bound.
 
Consider the virtual knot and Gauss diagram depicted in Figure \ref{Kishino-like2}. Clearly the diagram has bridge number 2, and we claim that the associated knot group $G_K$ is nontrivial and has meridional rank two. To see this, label the arcs of the Gauss diagram $a,b,c,d$ starting at the top. Then
\begin{eqnarray*}
G_K &\cong& \langle a,b,c,d \mid b=c a c^{-1}, c=a^{-1} b a, d=a c a^{-1}, a=c^{-1} d c \rangle \\
&\cong& \langle a,b \mid a^{-1} b a b^{-1}ab^{-1} \rangle.
\end{eqnarray*}
Using Fox derivatives, it follows that the first elementary Alexander ideal of $G_K$ is $\cE_1 =(1-2t)$, which is nontrivial and principal. Applying Proposition \ref{group-mr}, it follows that $\mr(G_K)>1$, and from this we conclude that  $\vb(K)=2.$

On the other hand, notice that $K=U \# U$ is the connected sum of two trivial knots, each having bridge number 1. Thus 
$\vb(K)=\vb(U\# U)=\vb(U)+\vb(U)$ in this example, showing that the formula \eqref{connsum} does not extend virtual knots. This gives a negative answer to Question 6.5 of \cite{HKK}.  

Arguing just as above, one can show in general that the welded bridge number satisfies
$$\wb(D_1\# D_2) ~\leq~\vbr(D_1)+\vbr(D_2),$$
and the  knot in Figure \ref{Kishino-like2} gives an example with $\wb(K_1\# K_2) = \wb(K_1)+\wb(K_2)$. Thus, the equation \eqref{connsum} does not generally hold for welded knots either.

\begin{figure}[ht]
\centering
\includegraphics[scale=1.80]{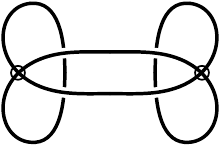} \qquad \qquad \includegraphics[scale=0.80]{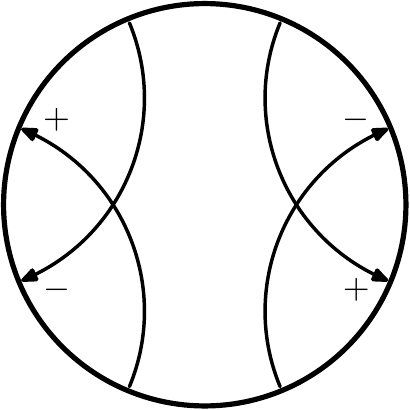}
\caption{A composite virtual knot with $\vb(K)=2$ and its Gauss diagram.}
\label{Kishino-like2}
\end{figure}

Now consider the virtual knot $K_n$ with Gauss diagram $D_n$ given in Figure \ref{Kishino-n}. Note that $\vbr(D_n) = n$, and in case $n=2$ it is just the diagram for the virtual knot in Figure \ref{Kishino-like2}. 
Labeling the over crossings $a_1, \ldots, a_n,$ one can show that
\begin{eqnarray*}
G_{K_n} &\cong& \langle a_1, \ldots, a_n \mid a_i a_{i+1}^{-1}a_i a_{i+1} a_i^{-1}a_{i+1}^{-1} \quad \text{for $1\leq i \leq n$} \rangle,
\end{eqnarray*}
where the subscripts on $a_i$ are taken modulo $n$.
Using Fox differentiation, we compute the Alexander matrix
$$A=\begin{bmatrix} 2-t & t-2 & 0 &\cdots & 0\\
0 & 2-t & t-2 &  &\vdots\\
\vdots& & \ddots & \ddots & 0\\
0&&&2-t& t-2 \\
t-2 & 0 & \cdots &0& 2-t
\end{bmatrix},$$
and using this it follows that $K_n$ has $(n-1)$-st elementary Alexander ideal given by $\cE_{n-1} =(1-2t)$. This ideal is evidently principal and nontrivial, and so Proposition \ref{group-mr} implies that $\vb(K_n) = \wb(K_n) = n.$

\begin{figure}[ht]
\centering
\includegraphics[scale=1.10]{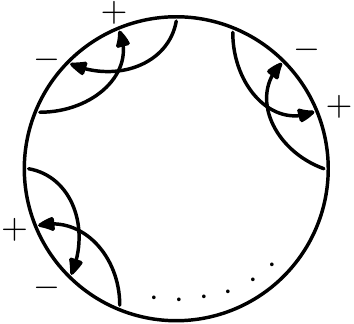}
\caption{The Gauss diagram of a virtual knot $K_n$ with $\vb(K_n)=n$.}
\label{Kishino-n}
\end{figure}

We now consider the operation of taking the vertical mirror image of a virtual knot.
Given a Gauss diagram $D$, let $D^*$ be the diagram obtained by reversing all the arrows and changing each of their signs.
 In terms of a virtual knot diagram for $K$, this corresponds to exchanging over and undercrossing arcs at each classical crossing. This operation induces an involution on virtual knots called vertical mirror symmetry, and we denote by $K^*$ the virtual knot obtained from $K$ by vertical mirror symmetry. 

\begin{thm} \label{thm2}
If $K$ is a virtual knot, then $\vb(K) = \vb(K^*)$. 
\end{thm}

\begin{proof}
To see this, note that given any Gauss diagram $D$ with $m$ bridges, then $D^*$ must also have $m$ bridges. Thus there is an isomorphism between the sets of Gauss diagrams representing $K$ and $K^*$, and this isomorphism preserves the bridge number of the underlying diagram. Since the virtual bridge number is defined as a minimum over all Gauss diagrams, this shows that  $\vb(K) = \vb(K^*).$
\end{proof}

Since the knot $K$ depicted in Figure \ref{Kishino-like2} has $\vb(K) = 2,$ 
it follows from Theorem \ref{thm2} that its vertical mirror image $K^*$ must also have
virtual bridge number $\vb(K^*) = 2.$ However, an elementary argument shows that $K^*$ is welded trivial, thus   $$1=\wb(K^*) ~< ~\vb(K^*)=2.$$
Applying the same reasoning to the family of knots $K_n$ in Figure \ref{Kishino-n}, it follows from Theorem \ref{thm2} that $K^*_n$ has
virtual bridge number $\vb(K^*_n) = n.$ Since $K^*_n$ is welded trivial,  $$1=\wb(K^*_n) ~< ~\vb(K^*_n)=n.$$
Thus, the family $K_n$ of knots shows that gap between the welded bridge number and the virtual bridge number can be arbitrarily large.

Notice that the knot groups $G_K$ and $G_{K^*}$ need not be isomorphic, and in fact the knot in Figure \ref{Kishino-like2} is an example where $G_K$ is nonabelian and $G_{K^*}$ is  infinite cyclic. By convention, $G_K$ is usually called the upper group and is denoted $G^+_K$, and $G_{K^*}$ is called the lower group and is denoted $G^-_K$. Both groups are invariant under virtual isotopy and have meridional rank less the virtual bridge number $\vb(K)$.
Thus 
\begin{equation} \label{G-bound}
\max\left\{ \mr(G^+_K), \mr(G^-_{K})\right\} ~\leq ~ \vb(K).
\end{equation}
In Section \ref{PandP}, we use parity to find a virtual knot $K$ for which the inequality \eqref{G-bound} is strict.

\section{The reduced virtual knot group}
In this section, we show how to derive bounds on the virtual bridge number from the reduced virtual knot group  $\RG_K$.
This is a group-valued invariant of virtual knots introduced in \cite{duality}. The group $\RG_K$ is isomorphic to two other the group valued invariants of virtual knots, namely the extended group $\wt{\pi}_K$ of Silver and Williams \cite{SW-Crowell} and the quandle group $\QG_K$ introduced by Manturov in \cite[Definition 3]{Manturov02}. For details on the proof of this result, we refer to \cite[Theorem 3.3]{duality}.  Note that $\QG_K$ coincides with the extended group of Bardakov and Bellingeri, cf. \cite[Definition 3]{BB14}.

Because $\RG_K$ is isomorphic to both of these other groups, the following application could be developed with any one of them.  We prefer to work with $\RG_K$ because it has the advantage of being computable directly from the Gauss diagram. The quandle group $\QG_K$ is most easily computed from a virtual knot diagram of $K$, but that requires information about the virtual crossings.

Given a Gauss diagram $D$, the reduced virtual knot group $\RG_D$ is the finitely presented group given by the following construction. Suppose $D$ has $n$ chords, and consider the generators $a_1, \ldots, a_{2n}$ given by the $n$ arcs of $D$ from one arrowhead or arrowfoot to the next. We order the chords $c_1, \ldots, c_n$ and arcs $a_1,\ldots, a_{2n}$  successively so that the $k$-th chord $c_k$ has overcrossing arcs $a_j, a_{j+1}$ and undercrossing arcs $a_i, a_{i+1}$. Here $i=i(k)$ and $j=j(k)$ are both functions of $k=1,\ldots, n.$ Given a chord $c_k$ as in Figure \ref{GD-VGK-bar-crossing}, we impose the relations $a_{j+1} =v^{\ep_k} a_j v^{-\ep_k}$ and $a_{i+1} = a_j^{\ep_k} v^{-\ep_k} \, a_i \, v^{\ep_k} a_{j}^{-\ep_k}$, where $\ep_k$ is the sign of chord $c_k$.

\begin{figure}[ht]
\centering
\includegraphics[scale=0.90]{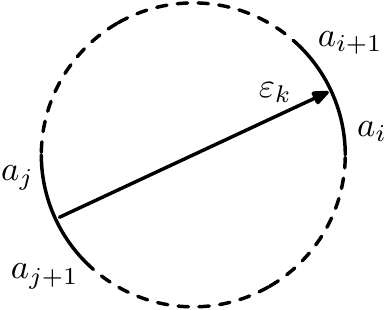}  
\caption{The chord $c_k$ gives Wirtinger relations $a_{j+1} =v^{\ep_k} a_j v^{-\ep_k}$ and $a_{i+1} = a_j^{\ep_k} v^{-\ep_k} \, a_i \, v^{\ep_k} a_{j}^{-\ep_k}$.}
\label{GD-VGK-bar-crossing}
\end{figure}

The group $\RG_D$ is invariant under virtual isotopy, and we refer to \cite{duality} for the proof. The resulting invariant is called the reduced virtual knot group of $K$ and will be denoted $\RG_K$.

Notice that for the unknot $U$, $\RG_U = \langle a,v \rangle.$ Therefore, we will say that $K$ has trivial reduced virtual knot group if $\RG_K$ is a free group on two generators.
The reduced virtual knot group $\RG_K$ is similar to the knot group $G_K$, except that it is not invariant under the forbidden move, so it generally carries more information than $G_K$, see \cite{BDGGHN, duality}. For example, there are infinitely many nontrivial virtual knots $K$ with virtual bridge number one, and each one is welded trivial and so has trivial knot group. Most of these knots have nontrivial $\RG_K$. 
 
We now discuss the notion of meridional rank for the reduced virtual knot group $\RG_K$. Given a presentation
\begin{equation}\label{quandle}
\RG_K = \langle a_1, \ldots , a_n, v \mid  r_1, \ldots, r_n  \rangle,
\end{equation} 
we say $a_i$ is a meridional generator if it comes from an arc of the knot diagram of $K$. We say the given presentation has meridional rank $n$, and we define the meridional rank of $\RG_K$ to be the minimum over all Tietze equivalent presentations of $\RG_K$.

Suppose now that $K$ is an $m$-bridge virtual knot. Arguing as before, $\RG_K$ admits a presentation with $m$ meridional generators. This shows that for any virtual knot $K$, the meridional rank and virtual bridge number satisfy 
\begin{equation} \label{quandlegp}
\mr(\RG_{K}) \leq \vb(K).
\end{equation}

If $K$ is a classical knot, or more generally almost classical, then Theorem 5.3 of \cite{duality} implies that  the reduced virtual knot group splits as a  a free product $\RG_K \cong  G_K * \langle  v \rangle  $ of the knot group with an infinite cyclic group.  For a general virtual knot, 
the quotient $\RG_K\! /\langle \! \langle v \rangle \! \rangle$ of the reduced virtual knot group is isomorphic to the knot group $G_K$. It follows from these observations that $\mr(G_K) \leq \mr(\RG_K)$. 
 
As with the knot group, both $\RG_K$ and $\RG_{K^*}$ are invariant under virtual isotopy and have meridional rank less the virtual bridge number $\vb(K)$.
Thus 
\begin{equation} \label{RG-bound}
\max\left\{ \mr(\RG_K), \mr(\RG_{K^*})\right\} ~\leq ~ \vb(K).
\end{equation}
Notice that $\max\left\{ \mr(G_K), \mr(G_{K^*})\right\} \leq \max\left\{ \mr(\RG_K), \mr(\RG_{K^*})\right\}$, so the reduced virtual knot groups give a better lower bound for  the virtual bridge number $\vb(K)$.
It would be interesting to find an example of a virtual knot $K$ where the inequality \eqref{RG-bound} is strict.  

One can use the elementary ideal theory of the reduced virtual knot group $\RG_K$ to deduce information about its meridional rank $\mr(\RG_K).$ This is similar to the discussion for $G_K$, the only real difference is that $\RG_K$ has abelianization   
isomorphic to the free abelian group $\ZZ^2$ of rank two,
and as a result $\RG_K'/ \RG_K''$ becomes a module over the ring $\ZZ[t^{\pm 1}, v^{\pm 1}]$. As before, one can easily determine a presentation matrix for this module using Fox differentiation.
For example, given the group presentation \eqref{quandle},
the associated presentation matrix  is the $(n+1) \times n$ matrix 
$$M = \begin{bmatrix} A & \left(\frac{\partial r^{}}{\partial v_{}} \right)  \end{bmatrix},$$
where
$A$
is the $n \times n$ Alexander matrix given in \eqref{Alex} and
$\left(\frac{\partial r^{}}{\partial v_{}} \right)$
is the column vector with $i$-th entry equal to $ \left. \frac{\partial r_i}{\partial v}\right|_{a_1, \ldots, a_n=t}$.

The $k$-th elementary ideal  is the ideal generated by all $(n-k) \times (n-k)$ minors of $M$, and they are invariants of the underlying module and form a nested  sequence 
\begin{equation} \label{chain2}
 (0)\subset \wbar{\cE}_0 \subset \wbar{\cE}_1 \subset \cdots \subset \wbar{\cE}_n  = \ZZ[t^{\pm 1}, v^{\pm 1}].
 \end{equation}
Note that because the presentation matrix $M$ is no longer square, the $0$-th order ideal $\wbar{\cE}_0$ is not necessarily trivial. Just as before, the length of the chain \eqref{chain2} of ideals is bounded by the meridional rank of $\RG_K$.
Indeed, if $\RG_K$ has meridional rank $k$, then using a presentation of $\RG_K$ with $k$ meridional generators, it follows that $\wbar{\cE}_k = (1)$. 
The following proposition summarizes our discussion.
\begin{prop} \label{quandle-mr}
Let $K$ be a virtual knot and consider the elementary ideal theory associated to the reduced virtual knot group $\RG_K$. If the $k$-th elementary ideal $\cE_k \neq (1)$ is proper, then the reduced virtual knot  group $\RG_K$ has meridional rank $\mr(\RG_K) \geq k+1$ and $K$ has virtual bridge number $\vb(K) \geq k+1$.
\end{prop}


\section{Kosaka's conjecture}

In Kamada's notation for virtual knots, the sequence of over and under crossings is called the primary data of a knot diagram, and the signs of all the crossings constitute the secondary data. In terms of Gauss diagrams, the primary data consists of the (unsigned) chords, and the secondary data consists of their signs. Turaev studies equivalence classes of unsigned Gauss diagrams and calls them pseudo-knots. For instance, for a classical knot diagram, there are only two realizable signatures on the associated pseudo-knot. The situation is different in the virtual realm, where there are $2^n$, where $n$ is the number of crossings.

Problem 6.3 in \cite{HKK} is the following conjecture.

\begin{conj} \emph{(Kosaka)} Virtual knots have the same virtual bridge number if they have the same primary data.
\end{conj}

There are two essential ways to interpret this conjecture, a weak form and a stronger form. We shall see that the weak form is true and the strong form is false.

The first interpretation asks whether Kosaka's conjecture is true on the level of Gauss diagrams. We call this the weak Kosaka conjecture, and it is true since changing the sign of a chord has no effect on the bridge number of a Gauss diagram $D$. 

\begin{figure}[h]
\centering
\includegraphics[scale=0.80]{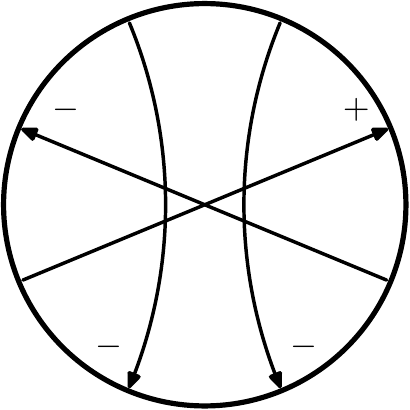} \qquad \qquad \includegraphics[scale=0.80]{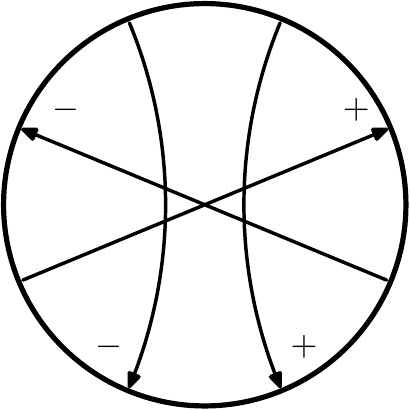}
\caption{On the left is the Gauss diagram of $K=4.103$, which has $\vb(K)=2$. On the right is a related Gauss diagram that is virtually trivial.}
\label{4-103}
\end{figure}

The second interpretation asks whether Kosaka's conjecture is true on the level of virtual knots, and we shall show this is false by presenting a counterexample. Let $K$ be the virtual knot $4.103$ from \cite{Green}. 

First, we claim that is $G_K$ has meridional rank two. To see this, labelling the arcs of the Gauss diagram $a,b,c,d$ starting at the top, it follows that
\begin{eqnarray*}
G_K &\cong& \langle a,b,c,d \mid b=d^{-1} a d, c=a^{-1} b a, d=a^{-1} c a, a=b d b^{=1} \rangle \\
&\cong& \langle a,b \mid a^2 b^{-1}aba^{-2}b^{-1} \rangle.
\end{eqnarray*}
Using Fox derivatives, it follows that the first elementary Alexander ideal of $G_K$ is $\cE_1 =(1+t-t^2)$, which is nontrivial and principal. Using Proposition \ref{group-mr}, it follows that $\mr(G_K)=2$ and thus $\vb(K) \geq \mr(G_K) =2.$

In fact, in \cite{HKK} the authors show that $K$ admits a two-bridge diagram, and so $\vb(K)=2$. On the other hand, changing the sign of one of its chords  gives a diagram that is virtually trivial. To see this, apply two Reidemeister 2 moves to the Gauss diagram depicted on the right of Figure \ref{4-103}. If $K'$ denotes the associated virtual knot, since $K'$ is trivial, we have $\vb(K') = 1,$ and this shows that Kosaka's conjecture is not true on the level of virtual knots.   

\section{Parity, projection, and virtual bridge number} \label{PandP}
In this section, we introduce Manturov's concept of parity for virtual knots and use Gaussian parity to provide an example of a virtual knot $K$ with $\vb(K)>1$ and trivial upper and lower knot groups, answering Question 6.2 of \cite{HKK}.

A parity is a function $f$ that assigns to each crossing $c$ in a virtual knot diagram $D$ an element of $\ZZ/2$ satisfying the following axioms: 

\begin{enumerate}[(i)]
\item The sum of the parity of all crossings in a Reidemeister 1, 2, or 3 move equals $0 \mod 2$. 
\item The parity of each crossing in a Reidemeister 3 move is preserved. 
\end{enumerate}

The crossing $c$ is called \emph{even} if $f(c)=0$ and \emph{odd} if $f(c)=1.$ Given a parity function $f$, one can define a projection $P_f$, which consists of turning each odd crossing of $D$ into a virtual crossing. On the level of Gauss diagrams, this corresponds to erasing the odd chords. The next result contains the key observation that relates the parity axioms with the Reidemeister moves, and this was originally proved by Manturov in \cite{Manturov}.
\begin{lem}
\label{lemma-Manturov}
If two diagrams $D_1$ and $D_2$ are virtually isotopic, then so are their projections $P_f(D_1)$ and $P_f(D_2)$. Consequently, $P_f$ is well-defined as a map of virtual knots $K$.
\end{lem}
 
\begin{proof}
It is enough to show this in the case where $D_1$ and $D_2$ are related by a single Reidemeister move. In the case of a  Reidemeister 1 move, the parity axioms imply the crossing must be even so it is left unchanged under $P_f.$ For a Reidemeister 2 move, either both crossings are even, and nothing changes, or they are both odd, and a virtual  Reidemeister 2 move shows that $P_f(D_1)$ is virtually isotopic to $P_f(D_2).$ For a Reidemeister 3 move, either all three crossings are even and everything follows as before, or two of them are odd, in which case the mixed move can be applied to show $P_f(D_1)$ is virtually isotopic to $P_f(D_2).$ 
\end{proof}

Parities and their projections were applied by Manturov to show that classical knots obtain their minimal crossing number  on classical diagrams \cite{Manturov}, and Chernov observed that the same approach shows that classical knots obtain their minimal bridge number on classical diagrams \cite{Chernov}, see also Theorem \ref{manturov}. This result is a consequence of the following theorem. 

\begin{thm}
\label{proj}
For any parity function $f$, if $K$ is a virtual knot and $K' = P_f(K)$ is its projection, then $\vb (K')\le \vb (K)$. 
\end{thm}

\begin{proof}
Suppose $D$ is a virtual knot diagram for $K$ with minimal bridge number, and let $D'$ be the virtual knot diagram obtained from $D$ by replacing each odd crossing with a virtual crossing. Thus $D'$ represents the virtual knot $K'= P_f(K)$. Since $D'$ is obtained from $D$ by deleting crossings, it follows that $\vbr(D')\leq \vbr(D)$. Thus $\vb(K') \leq \vbr(D')\leq \vbr(D) = \vb(K)$. 
\end{proof}

The Gaussian parity $f$ is defined as follows. For a chord $c$ in a Gauss diagram $D$, let $f(c)$ be the number of other chords of $D$ that link $c$ modulo two. To be concrete, $f(c)$ is obtained by counting, mod two, the arrowheads and arrowtails along the arc of $D$ connecting the arrowtail of $c$ to its arrowhead. 

We apply Gaussian parity to the knot $K$ depicted on the left of Figure \ref{GD-ex}. This diagram has four odd chords, and removing them gives the diagram on the right of Figure \ref{GD-ex} representing the virtual knot $K'=3.7.$ It is not difficult to see that $\vb(K')=2,$ and Theorem \ref{proj} implies that $\vb(K) \geq 2.$ Furthermore, one can show that $K$ is welded trivial, and that $K^*$ is also welded trivial.  It follows that $K$ is a knot with trivial upper and lower knot groups. 
Thus, Proposition 3.4 of \cite{HKK} implies that $K$ has trivial upper and lower quandles, and this answers Question 6.2 of \cite{HKK}.

\begin{figure}[ht]
\centering
\includegraphics[scale=0.90]{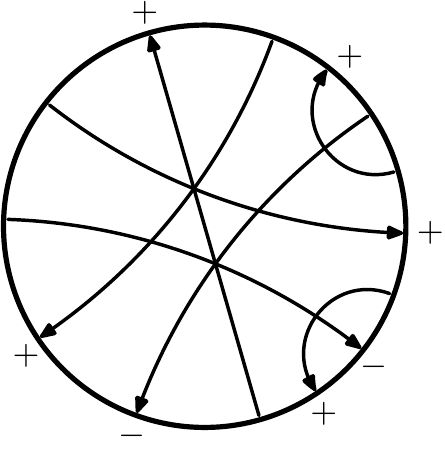}  \qquad \qquad
\includegraphics[scale=0.90]{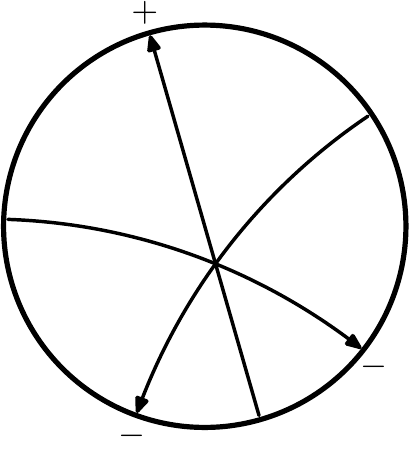} 
\caption{The Gauss diagram of a virtual knot $K$ on left and its projection $K'=P_f(K)$ on right.}
\label{GD-ex}
\end{figure}

\medskip

\noindent
{\it Acknowledgements.} We would like to thank Lindsay White for her help in computing the group invariants of several virtual knots. We would also like to thank Eric Harper and Andrew Nicas for their valuable input on this paper. H. Boden was supported by a Discovery Grant from the Natural Sciences and Engineering Research Council of Canada, and
A. Gaudreau was supported by a Postgraduate Scholarship from  the Natural Sciences and Engineering Research Council of Canada.

\end{document}